\documentclass[11pt]{article}
\usepackage{epsfig,amsmath,amssymb,latexsym,color}
\usepackage{hyperref}
\usepackage{algorithm}
\usepackage{algorithmic}
\usepackage{multirow}
\usepackage{wrapfig,lipsum,booktabs}
\usepackage{xcolor}%
\usepackage{textcomp}%
\usepackage{manyfoot}%
\usepackage{booktabs}%
\usepackage{algorithm}%
\usepackage{algorithmic}
\usepackage{listings}%
\usepackage{wrapfig,lipsum,booktabs}

\newcounter{mnote}
  \let\oldmarginpar\marginpar
 \renewcommand\marginpar[1]{\-\oldmarginpar[\raggedleft\footnotesize #1]%
    {\raggedright\footnotesize #1}}

\def\bfb{{\bf b}}

\def\vol{{\rm vol}}

\def\bfx{{\bf x}}

\def\det{{\rm det}}

\newtheorem{theorem}{Theorem}
\newtheorem{lemma}{Lemma}

\newtheorem{corollary}{Corollary}
\newtheorem{remark}{Remark}

\newtheorem{definition}{Definition}
\newenvironment{proof}{\begin{trivlist}\item[]{\emph{Proof.}}}
               {\hfill$\Box$\end{trivlist}}

\begin{document}
\title{Maximal Volume Matrix Cross Approximation for Image Compression and Least Squares Solution} 
\author{Kenneth Allen\footnote{kencubed@gmail.com. Department of Mathematics, University of Georgia, Athens, GA 30602}
\and
Ming-Jun Lai~\footnote{mjlai@uga.edu. Department of Mathematics,
University of Georgia, Athens, GA 30602. This author is supported by the Simons Foundation 
Collaboration Grant \#864439.} 
\and 
Zhaiming Shen\footnote{zshen49@gatech.edu. School of Mathematics,
Georgia Institute of Technology, Atlanta, GA 30332.}}
\maketitle
\begin{abstract} 
We study the classic matrix cross approximation based on the maximal volume submatrices. Our main results consist of an improvement of the classic estimate for matrix cross approximation and a greedy approach for finding the maximal volume submatrices. More precisely, we present a new proof of the classic estimate of the inequality with an improved constant. Also, we present a family of greedy maximal volume algorithms to improve the computational efficiency of matrix cross approximation. The proposed algorithms are shown to have theoretical guarantees of convergence. Finally, we present two applications: image compression and the least squares approximation of continuous functions. Our numerical results at the end of the paper demonstrate the effective performance of our approach.
\end{abstract} 


\section{Introduction}\label{sec1}

Low rank matrix approximation has been one of the vital topics of research for decades. It has a wide range of applications in the areas such as signal processing \cite{Cai20,Candes07}, matrix completion \cite{Cai23,Candes12,Wang15}, data compression \cite{Fazel08,Mitrovic13}, just to name a few. One obvious way to obtain such an approximation is to use singular value decomposition (SVD), which has been the most commonly used technique in low rank approximation. However, the drawbacks of SVD are that it can only express data in some abstract basis (e.g. eigenspaces) and hence may lose the interpretability of the physical meaning of the rows and columns of a matrix. Also, the computation of eigen-decomposition of matrices can be burdensome for large-scale problems.

It is also common to use  low rankness for compressing images. For example, if a black-and-white image is of rank 1, we only need to use one column and one row to represent this image instead of using the entire image. Although one can use SVD to reduce a matrix into small blocks, the integer property of the image will be lost. That is, the decomposed image is no longer of integer entries which will take more bits for an encoder to compress it. It is desired to have a low-rank integer matrix to approximate an integer image. One way to do it is to use the so-called matrix cross approximation as explained in this paper.

Another application of the matrix cross approximation is to solve the 
least squares approximation problem which is commonly encountered in practice. More precisely, let us 
say we need to solve the following least squares problem:
\begin{equation} \label{lsq}
    \min_{\mathbf{x}}\|A\mathbf{x}-\mathbf{b}\|_2
\end{equation}
for a given matrix $A$ of size $m\times n$ and the right-hand side $\mathbf{b}$. When $m\gg n$ is very
large and/or the singular value $\sigma_n(A)\approx 0$, one way to simplify the computation is to approximate $A$ by its matrix cross approximation and then solve a much smaller size least squares problem or a better conditioned least squares problem. This technique is particularly useful when we  deal with multiple least squares tasks with same data matrix and different response vectors. It helps to speed up the computation without sacrificing the interpretability of least squares parameters.

Let us recall the matrix cross approximation as follows. Given a matrix $A \in \mathbb{R}^{m\times n}$, letting  $I = \{i_1,\cdots,   i_k\}$ and $J = \{j_1, \cdots,  j_k\}$ 
be the row and column indices,   we denote by $A_{I,:}$  the all the rows of 
$A$ with  row indices in $I$ and $A_{:,J}$  the all the columns of $A$ with column indices in $J$. Assume that 
the submatrix   $A_{I;J}$ situated on rows $I$ and columns $J$ is nonsingular.  Then 
\begin{equation}
\label{skeleton}
A_{:,J} A_{I,J}^{-1} A_{I,:} \approx A
\end{equation}
is called a skeleton or cross approximation of $A$. The left-hand side of (\ref{skeleton}) is also called CUR decomposition in the literature.
Note that if $\#(I)=\#(J)=1$, it is the classic skeleton approximation. 
There is  plenty of study on the cross approximation and CUR decomposition of matrices. See \cite{GH17,Goreinov97,GOSTZ08,GT01,Hamm20,KS16,Mahoney09,MO18}
and the literature therein. 
The study is also generalized to the setting of tensors.  See, e.g. \cite{Cai21,Mahoney06,OT10}.  

There are different strategies to obtain such an approximations or decomposition in the literature. These include volume optimization \cite{Goreinov97, GOSTZ08, GT01}, rank and spectrum-revealing factorization \cite{gu1996efficient,anderson2017efficient}, random sampling \cite{Cai23, Hamm20b}, leverage scores computing \cite{Boutsidis14, Drineas08, Mahoney09}, and empirical interpolation \cite{Chaturantabut10, Drmac16, Sorensen16}. The approach we will be focusing on falls into the first category.

To state the main results in the paper, let us recall the definition of Chebeshev norm. For any matrix $A=[a_{ij}]_{1\leq i\leq m, 1\leq j\leq n}$, $\|A\|_C=\max_{i,j}|a_{ij}|$ is called Chebyshev norm for matrix $A$. Also, the volume of a square matrix $A$ is just the absolute value of the determinant of $A$. We start with an error estimate of the matrix cross approximation from \cite{GT01}.
\begin{theorem} [cf. \cite{GT01}]
\label{GT01}
Let $A$ be a matrix of size $m\times n$. Fix $r \in [1, \min\{m,n\})$. Suppose that $I \subset \{1,\cdots,m\}$ with $\#(I)=r$ and 
$J\subset \{1, \cdots, n\}$ with $\#(J)=r$ such that 
$A_{I,J}$ has the maximal volume among all $r \times r$ submatrices of $A$.  
Then the Chebyshev norm of the residual matrix satisfies the following estimate: 
\begin{equation}
\label{skapp}
\|A - A_r\|_C \le (r + 1)\sigma_{r+1}(A),
\end{equation}
where $A_r= A_{:,J} A_{I,J}^{-1} A_{I,:}$ is the $r$-cross approximation of $A$, 
$\sigma_{r+1}(A)$ is the $(r+1)^{th}$ singular value of $A$.
\end{theorem}

The above result was established in \cite{GT01}. Recently, there were few studies to give better estimate than the one in (\ref{skapp}) or 
 or generalization of the estimate in (\ref{skapp}). See \cite{HH23,osinsky2018pseudo} and the references therein. 
We are also interested in a better estimate than the one in (\ref{skapp}). 
It turns out that we are able to give another proof of the result in \cite{HH23} 
based on the properties of matrices and a recent result about the 
determinant of matrix
(cf. \cite{GH17}). That is, we are able to improve this error bound from (\ref{skapp}) to (\ref{myskapp}). 
It is clear that a submatrix with the maximal volume as described in 
Theorem~\ref{GT01} is hard to find. There are a few algorithms 
available in the literature. 
In this paper, we propose a family of greedy algorithms to speed up the computation of submatrices with the maximal volume which has a theoretical guarantee of convergence. Finally, in the paper, 
we present two applications:
one is for image compression and the other one is for least squares approximation. These two applications are standard
for any matrix computational algorithms. See, e.g. \cite{anderson2017efficient}.   
Indeed, when $A$ is an image with integer entries,  one can encode 
$A_{I,:}, A_{:,J}$ instead of the entire image $A$ to send/store the 
image.  After receiving integer matrices $A_{I,:}$ and $A_{:,J}$, 
one can decode them and compute according to the left-hand side of (\ref{skeleton})  which will be a good approximation of $A$ 
if $I$ and $J$ are appropriately chosen by using the greedy algorithms to be discussed in this paper.
We shall present a family of greedy algorithms for fast computing the maximal volume of submatrices of $A$ for finding these row and column indices $I$ 
and $J$. Their convergences and numerical performances will also
be discussed in a later section.  

Next we consider a least squares problem of large size in (\ref{lsq}) with matrix $A$ 
which may not be full rank, i.e. $\sigma_n(A)=0$. If $\sigma_{r+1}(A)\approx 0$ for $r+1\leq n$ while $\sigma_r(A)$
is large enough, we can use the matrix cross approximation $A_{:,J}A_{I,J}^{-1}A_{I,:}$ to approximate $A$ with $\#(I)=\#(J)=r$. Then the least squares problem of large size can be solved by finding
\begin{equation}
\min_{\hat{\mathbf{x}}}\|A_{:,J}A_{I,J}^{-1}A_{I,:}\mathbf{\hat{x}}-\mathbf{b}\|_2.
\end{equation}
Note that letting $\hat{\mathbf{x}}$ be the new 
least squares solution, it is easy to see 
\begin{equation}
\|A_{:,J}A_{I,J}^{-1}A_{I,:}\mathbf{\hat{x}}-\mathbf{b}\|_2
\leq \|A\mathbf{x}_b-\mathbf{b}\|_2+\|(A_{:,J}A_{I,J}^{-1}A_{I,:}-A)\mathbf{x}_b\|_2,
\end{equation} 
where $\mathbf{x}_b$ is the original least squares solution. The first term on the right-hand side is the norm of the standard least squares residual vector while the second term is nicely bounded above if $A_{:,J}A_{I,J}^{-1}A_{I,:}$ is a good approximation of $A$, i.e., $\sigma_{r+1}(A)\approx 0$. We shall discuss more towards this aspect in the last section.

\section{Improved error bounds}
We now see that the matrix cross approximation  is useful and the approximation in (\ref{skapp}) is very crucial. It is important to improve its estimate. Let us recall the
following results to explain the cross decomposition, i.e., an estimate of the residual matrix $E= A - A_{:,J} A_{I,J}^{-1} A_{I,:}$.

For $I=[i_1, \cdots, i_r]$ and $J=[j_1, \cdots, j_r]$, let 
\begin{equation}
\label{eij}
\mathcal{E}_{ij} = \left[
\begin{matrix} A_{ij} & A_{i,j_1} & \cdots & A_{i, j_r}\cr 
A_{i_1,j} & A_{i_1,j_1} & \cdots & A_{i_1,j_r}\cr 
\vdots & \vdots & \ddots & \vdots \cr
A_{i_k,j} & A_{i_k,j_1} & \cdots & A_{i_k,j_r}\cr
\end{matrix} \right]
\end{equation}
for all $i=1, \cdots, m$ and $j=1, \cdots, n$.  The researchers in \cite{GH17} proved the following.
\begin{theorem} [cf. \cite{GH17}]
\label{SkeletonError}
Let 
\begin{equation}
\label{eq:SkeletonError}
E= A - A_{:,J} A_{I,J}^{-1} A_{I,:}. 
\end{equation}
Then each entry $E_{ij}$ of $E$ is given by 
\begin{equation}
\label{key}
E_{ij}= \frac{ \det \mathcal{E}_{ij}}{\det A_{I,J}}
\end{equation}
for all $i=1, \cdots, m$ and $j=1, \cdots, n$. 
\end{theorem}

Notice that $\det(\mathcal{E}_{ij})=0$ for all $i\in I$ and for all $j\in J$. That is, the entries of $E$ are all zero except
for $(i,j), i\not\in I$ and $j\not\in J$. 
Since the SVD approximation is optimal, it is easy to see that a matrix cross approximation is not as good as using SVD in Frobenius norm. However, 
some cross approximations can be as good as using SVD if $A_{I,J}$ has a maximal volume which is defined by
\begin{equation}
\label{vol}
\hbox{vol}(A) =\sqrt{\det (A A^*)}
\end{equation}
for any matrix $A$ of size $m\times n$ with $m\le n$. Note that the concept of
the volume of a matrix defined in this paper is slightly different from the one in 
\cite{BI92}, \cite{P00}, and \cite{SG20}. In the 
literature mentioned previously, the volume is the product of all singular values of the matrix
$A$. However, the value of volumes in the two definitions is the same. The researchers in  \cite{P00}, and \cite{SG20} used the concept of the maximal volume to reveal the rank of a matrix.  

We turn our interest to matrix approximation. The estimate in (\ref{skapp}) is improved to (\ref{myskapp}) in this paper. Although in a different formulation, this result is the same as the one in \cite{HH23}. We shall give a proof which is simpler than the one in \cite{HH23} and derive a few corollaries based on the decaying property of singular values. We also point out in Remark~\ref{rmk} that the standard assumption made in our improved result can be further weakened.
\begin{theorem}
\label{mjlai06062021}
Under the assumption in Theorem~\ref{GT01}, we have the following estimate: 
\begin{equation}
\label{myskapp}
\|A - A_r\|_C \le \frac{(r+1)\sigma_{r+1}(A)}{\sqrt{1+\sum_{k=1}^{r}\frac{\sigma_{r+1}^2(A)}{\sigma_k^2(A)}}}.
\end{equation}
\end{theorem}

The improvement from (\ref{skapp}) to (\ref{myskapp}) is essential. For example, if the singular values do not decay too fast, say $\sigma_k(A)\approx 1/k$, then the bound in (\ref{skapp}) does not converge at all when $r\to\infty$, while the bound in (\ref{myskapp}) does converge.

It is worthwhile to mention that the authors in \cite{osinsky2018pseudo} gave  another estimate without the factor $r+1$ based on a larger cross matrix $A_r$ of size ${p\times q}$, $p\geq r$, $q\geq r$. Although their estimate can be as good as twice of $\sigma_{r+1}(A)$, it is not better in general. For example, if $p=q=r$, then their estimate is the same as (\ref{skapp}).

In the principle component analysis (PCA), a matrix may have a big leading singular value while the rest of the singular values are very small, that is, $\sigma_1(A)\gg \sigma_k(A), k\ge 2$. For such a matrix $A$, say $\sigma_2(A)$ is only half of $\sigma_1(A)$. Using the cross approximation with $r=1$,  
the estimate on the right-hand 
side of (\ref{myskapp}) is about $1.789\sigma_{2}(A)$, while the estimate on the right-hand side of (\ref{skapp}) is $2\sigma_2(A)$.   Note that the computation of the matrix cross approximation when $r=1$ is extremely simple which is an advantage over the standard PCA, the latter requires the power method to find the leading singular value and two singular vectors.  

From the result above, we can see 
$A_r= A_{:,J} A_{I,J}^{-1}A_{I,:}$ is a good approximation of $A$ similar to the SVD approximation of $A$ when $A_{I,J}$ is a maximal volume submatrix and the rank $r$ is not too big. However, finding $A_{I,J}$ is not trivial when $r>1$, although it is easy to see that a maximal volume submatrix $A_{I,J}$ always exists since there are finitely many submatrices of size 
$r\times r$ in $A$ of size $n\times n$. As $n\gg 1$, the total number of choices of submatrices of size $r\times r$ increases
combinatorially large.  Computing the $A_{I,J}$ is an NP hard problem (cf. \cite{CI09}). To help find such indices $I$ and $J$, the following concept of dominant matrices is introduced according to the literature (cf. \cite{GOSTZ08,GT01,MO18}).  

\begin{definition}
Let $A_{\square}$ be 
a $r \times r$ nonsingular submatrix  of the given rectangular $n \times r$ matrix $A$, where 
$n>r$.  $A_{\square}$ is  dominant  if all the entries of 
$A A^{-1}_{\square}$ are not greater than 1 in modulus.
\end{definition}

Two basic results on dominant matrices can be found in \cite{GOSTZ08}.  
\begin{lemma} 
\label{GOSTZ08}
For $n \times r$ matrix $A$, a maximum volume $r \times r$ submatrix $A_\diamondsuit$ is dominant.
\end{lemma}

On the other hand, a dominant matrix is not too far away from the maximal volume matrix $A_\diamondsuit$ when $r>0$ as shown in the following.
\begin{lemma}
For any  $n \times r$ matrix $A$ with full rank, we have
\begin{equation}
\label{lowbound}
|\det (A_{\square})|\ge |\det(A_\diamondsuit)|/r^{r/2}
\end{equation}
for all dominant $r \times r$ submatrices $A_\square$ of $A$, where $A_\diamondsuit$ is 
a submatrix of $r\times r$ with maximal volume.  
\end{lemma}

We shall use a dominant matrix to replace the maximal volume matrix $A_\diamondsuit$. In fact, in \cite{GT01}, the researchers
showed that if $A_{I,J}$ is a good approximation of $A_\diamondsuit$ with $\det(A_{I,J})=\nu \det(A_\diamondsuit)$ with $\nu\in (0,1)$, then 
\begin{equation}
\label{mjlai12062020}
\|A - A_{:,J} A_{I,J}^{-1} A_{I,:}\|_C \le \frac{1}{\nu} (r+1)\sigma_{r+1}(A).
\end{equation}
Our proof of Theorem~\ref{mjlai06062021} yield the following result, which gives an estimate when $A_{I,J}$ is a dominant matrix.
\begin{theorem}
\label{mjlai06082021}
Suppose that $A_{I,J}$ is a dominant matrix satisfying $|\det(A_{I,J})|=\nu 
|\det(A_\diamondsuit)|$ with $\nu\in (0,1]$. Then using the dominant matrix $A_{I,J}$ to 
form the cross approximation $A_r$, we have  
\begin{equation}
\label{myskapp2}
\|A - A_r\|_C \le \frac{(r+1)\sigma_{r+1}(A)}{\nu \sqrt{1+\sum_{k=1}^{r}\frac{\sigma_{r+1}^2(A)}{\sigma_k^2(A)}}}.
\end{equation}
\end{theorem}

The proofs of Theorems~\ref{mjlai06062021} and \ref{mjlai06082021} will be given in the next section. Based on Lemma~\ref{C=AB}, we will also give a different proof for Theorem~\ref{GT01} as an intermediate result. In addition
to the introduction of dominant submatrix, the researchers in \cite{GOSTZ08} also explained a computational approach called maximal volume (maxvol) algorithm, and several improvements of the algorithm were studied, e.g., see \cite{MO18,OT10}. We will study the maxvol algorithm in more detail in the section after the next.

\section{Proof of Theorems~\ref{mjlai06062021} and \ref{mjlai06082021}}
Let us begin with some elementary properties of the Chebyshev norm of matrices. Recall that for any matrix $A$ of size $(r+1)\times(r+1)$, we have $\|A\|_F=\sqrt{\sum_{i=1}^{r+1}\sigma_i^2(A)}$. Hence,
\begin{equation}
    \|A\|_F\geq \sqrt{r+1}\sigma_{r+1}(A) \quad\text{and}\quad \|A\|_F\leq \sqrt{r+1}\sigma_{1}(A).
\end{equation}
On the other hand, $\|A\|_2\leq \|A\|_F\leq\sqrt{(r+1)^2\|A\|^2_C}=(r+1)\|A\|_C$. It follows that
\begin{lemma}
    For any matrix $A$ of size $m\times n$, 
    \begin{equation}
        \sigma_1(A)\leq (r+1)\|A\|_C \quad \text{and}\quad \sigma_{r+1}(A)\leq \sqrt{r+1}\|A\|_C
    \end{equation}
\end{lemma}

Let us first prove Theorem~\ref{mjlai06062021}.

\begin{proof} [Proof of Theorem~\ref{mjlai06062021}] 
From (\ref{key}), we see that 
\begin{equation}
\label{startpt1}
\|E\|_C = \max_{i,j} \frac{|\det(\mathcal{E}_{ij})|}{|\det(A_{I,J})|}.
\end{equation}
We first note that the determinant $\det(\mathcal{E}_{ij})$ can be easily found. 
Indeed, let us recall the following  \textit{Schur determinant identity}:
\begin{lemma}
\label{det}
Let $A, B, C$, and $D$ be block sub-matrices of $M$ with $A$ invertible. Then 
$$\det(
\begin{bmatrix}
A & B \\
C & D
\end{bmatrix}) = \det(A) \det(D-CA^{-1}B).$$ 
In particular, when $A$ is an $r \times r$, $B$ and $C$ are vectors, and $D$ is a scalar,  
$$
\det(
\begin{bmatrix}
A & B \\
C & D
\end{bmatrix}) = (D-CA^{-1}B)\det(A).$$
\end{lemma}
\begin{proof}
Note that $\begin{bmatrix}
A & B \\
C & D
\end{bmatrix} =
\begin{bmatrix}
A & 0 \\
C & I
\end{bmatrix}
\begin{bmatrix}
I & A^{-1}B \\
0 & D-CA^{-1}B
\end{bmatrix}$. Taking the determinant we get the desired result.
\end{proof}

By using Lemma~\ref{det} and (\ref{eij}), we have
\begin{equation}
\label{result1}
\|E\|_C= \max_{i,j} |\det(A_{ij} - A_{i,J} A_{I,J}^{-1} A_{I,j})| 
=\max_{i,j} |A_{ij} - A_{i,J} A_{I,J}^{-1} A_{I,j}|.
\end{equation}

Next we need to estimate the right-hand side of (\ref{result1}) when $|\det(A_{I,J})|$ is largest. To do so, we consider
$\mathcal{E}_{ij}^{-1}$. Then it can be written as
\begin{equation} \label{det2}
    \mathcal{E}_{ij}^{-1} =\frac{1}{\det(\mathcal{E}_{ij})} [ \mathcal{E}_{ij}^*],
\end{equation}
where $\mathcal{E}_{ij}^*$ stands for the cofactors matrix of $\mathcal{E}_{ij}$. 
 It is easy to see that 
$\|\mathcal{E}_{ij}^{-1}\|_C = |\det(A_{I,J})|/|\det(\mathcal{E}_{ij})|$ 
by the definition of the cofactors matrix and  the $\det(A_{I,J})$ is the largest. 


By using Lemma~\ref{det}, we have
\begin{equation}
\label{result2}
\|\mathcal{E}_{ij}^{-1}\|_C = \frac{1}{|A_{ij} - A_{i,J} A_{I,J}^{-1} A_{I,j}|}.
\end{equation}

Next recall that for any matrix $A$ of size $(r+1)\times (r+1)$, we have 
$\|A\|_F = \sqrt{\sum_{i=1}^{r+1}\sigma_i^2(A)}$ and hence, 
\begin{equation}
\label{resultn1}
 \|A\|_F  \ge \sqrt{r+1}\sigma_{r+1}(A) \hbox{ and }
\|A\|_F\le \sqrt{r+1}\sigma_1(A).  
\end{equation} 
On the other hand, $\|A\|_2\le \|A\|_F \le \sqrt{ (r+1)^2 \|A\|_C^2} = (r+1)\|A\|_C.$ 
It follows
\begin{equation}
\label{result3}
\sigma_{1}(A)\le (r+1) \|A\|_C \hbox{ and } 
\sigma_{r+1}(A)\le \sqrt{r+1}\|A\|_C.
\end{equation}  
Furthermore, we have 
\begin{lemma}
\label{C=AB}
Let $B$ be a matrix of size $m\times n$ with singular values $s_1\ge s_2\ge \cdots \ge s_m$ 
and let $C=AB$ with singular values $t_1\ge \cdots \ge t_m$. Then 
$$
t_i\le s_i \|A\|_2, \quad i=1, \cdots, n
$$ 
\end{lemma}
\begin{proof}
Consider the eigenvalues of $C^\top C= B^\top A^\top A B$. 
Let $A^\top A= Q \Sigma^2 Q^\top$ be the spectral decomposition of $A^\top A$. Let 
$D=Q(\|A\|_2^2 I- \Sigma^2)Q^\top$. Then $D$ is positive semi-definite. Since $D= \|A\|_2 I - A^\top A$, we have
$$
\|A\|_2^2 B^\top B =  B^\top (A^\top A+ D)B= C^\top C+ B^\top D B\ge C^\top C.
$$
It follows that the eigenvalue $t_i^2$ of $C^\top C$ is less than or equal to $\|A\|_2^2 s_i^2$ for each $i$. 
\end{proof}

As $I_{r+1}= \mathcal{E}_{ij}\mathcal{E}_{ij}^{-1}$, we use Lemma~\ref{C=AB} to get
\begin{eqnarray}
1 &\le& \sigma_{r+1} (\mathcal{E}_{ij}) \|\mathcal{E}_{ij}^{-1}\|_2 
\le  \sigma_{r+1} (\mathcal{E}_{ij}) \|\mathcal{E}_{ij}^{-1}\|_F  
\le  \sigma_{r+1} (\mathcal{E}_{ij}) (r+1)\|\mathcal{E}_{ij}^{-1}\|_C. 
\end{eqnarray}
It follows that
\begin{eqnarray*} 
\frac{1}{\|\mathcal{E}_{ij}^{-1}\|_C} \le (r+1) \sigma_{r+1}(\mathcal{E}_{ij}). 
\end{eqnarray*}
Combining with the result in (\ref{result2}), we have
\begin{equation} \label{ptest}
    |A_{ij} - A_{i,J} A_{I,J}^{-1} A_{I,j}|\le (r+1) \sigma_{r+1}(\mathcal{E}_{ij}). 
\end{equation}
Using (\ref{result1}), we finish a proof of Theorem~\ref{GT01} 
since $\sigma_{r+1}(\mathcal{E}_{ij}) \le \sigma_{r+1}(A)$. 


Again, we use Lemma~\ref{C=AB} to get
\begin{eqnarray} \label{prodest}
1 &\le& \sigma_{r+1} (\mathcal{E}_{ij}) \|\mathcal{E}_{ij}^{-1}\|_2 
=  \sigma_{r+1} (\mathcal{E}_{ij}) \|\mathcal{E}_{ij}^{-1}\|_F 
\sqrt{1- \frac{\sum_{k=2}^{r+1}
\sigma_k^2(\mathcal{E}_{ij}^{-1})}{\sum_{k=1}^{r+1}\sigma_k^2(\mathcal{E}_{ij}^{-1})}}
\cr
&\le& \sigma_{r+1}(\mathcal{E}_{ij}) (r+1)\|\mathcal{E}_{ij}^{-1}\|_C\sqrt{ \frac{\sigma_1^2(\mathcal{E}_{ij}^{-1})}{\sum_{k=1}^{r+1}\sigma_k^2(\mathcal{E}_{ij}^{-1})}} \cr
&=& (r+1)\|\mathcal{E}_{ij}^{-1}\|_C\sqrt{ \frac{1}{\sum_{k=1}^{r+1}\sigma_k^2(\mathcal{E}_{ij}^{-1})}}.
\end{eqnarray}

The last equality holds since $\sigma_1(\mathcal{E}_{ij}^{-1})=1/{\sigma_{r+1}(\mathcal{E}_{ij})}$. 
This implies 
\begin{eqnarray*}
    \frac{1}{\|\mathcal{E}_{ij}^{-1}\|_C} \le \frac{r+1}{\sqrt{\sum_{k=1}^{r+1}\sigma_k^2(\mathcal{E}_{ij}^{-1})}}=\frac{r+1}{\sqrt{\sum_{k=1}^{r+1}\sigma_k^{-2}(\mathcal{E}_{ij})}}.
\end{eqnarray*}
Combining with the result in (\ref{result2}), we have
\begin{eqnarray*}
    |A_{ij} - A_{i,J} A_{I,J}^{-1} A_{I,j}|\le \frac{r+1}{\sqrt{\sum_{k=1}^{r+1}\sigma_k^{-2}(\mathcal{E}_{ij})}}\le \frac{r+1}{\sqrt{\sum_{k=1}^{r+1}\sigma_k^{-2}(A)}}=\frac{(r+1)\sigma_{r+1}(A)}{\sqrt{1+\sum_{k=1}^{r}\frac{\sigma_{r+1}^2(A)}{\sigma_k^2(A)}}}. 
\end{eqnarray*}
The second inequality in the above holds because of the singular value interlacing relation. For example, see Theorem 1 in \cite{thompson1972principal}. 
This finishes the proof of Theorem~\ref{mjlai06062021}.

\end{proof}

\begin{remark} \label{rmk}
It is worthwhile to point out that the assumption we made on the index set $I$ and $J$ for Theorem~\ref{mjlai06062021} can be further weakened. In (\ref{det2}) and (\ref{result2}), we do not need to assume $A_{I,J}$ has the global maximal volume, as long as $A_{I,J}$ has the largest volume among all the submatrix $A_{\Tilde{I},\Tilde{J}}$ where $\#(\Tilde{I})=r+1$, $\#(\Tilde{J})=r+1$, $I\subset \Tilde{I}$, $J\subset\Tilde{J}$, we are able to deduce the same result.
\end{remark}

 \begin{remark}
In fact our proof of
the main theorem gives a better estimate: Let $i_0,j_0$ be the row and column index which achieves the maximum in (\ref{result1}). Then

\begin{eqnarray*}
\|A-A_r\|_{C}= 
    |A_{i_0j_0} - A_{i_0,J} A_{I,J}^{-1} A_{I,j_0}|\le 
    \frac{r+1}{\sqrt{\sum_{k=1}^{r+1}\sigma_k^{-2}(\mathcal{E}_{i_0j_0})}}<\frac{r+1}{\sqrt{\sum_{k=1}^{r+1}\sigma_k^{-2}(A)}}
    \end{eqnarray*}
as it is probable to have 
$\sigma_k(\mathcal{E}_{i_0j_0})<
\sigma_k(A)$ for some $k$.
 \end{remark}

\begin{proof} [Proof of Theorem~\ref{mjlai06082021}]
Same as (\ref{det2}), we have
$\mathcal{E}_{ij}^{-1} =\frac{1}{\det(\mathcal{E}_{ij})} [ \mathcal{E}_{ij}^*]$, where $\mathcal{E}_{ij}^*$ stands for the cofactors matrix of $\mathcal{E}_{ij}$. 
 It is easy to see that 
\begin{equation}
 \|\mathcal{E}_{ij}^{-1}\|_C = \frac{|\det(A_{I_i,J_j})|}{|\det(\mathcal{E}_{ij})|}\le \frac{|\det(A_{\diamondsuit})|}{|\det(\mathcal{E}_{ij})|}=\frac{|\det(A_{I,J})|}{\nu|\det(\mathcal{E}_{ij})|},
\end{equation}
where $I_i$ and $J_j$ corresponds to the row and column index set which has the largest determinant among $\mathcal{E}_{ij}$. By Lemma~\ref{det}, we have
\begin{equation}
\frac{1}{\|\mathcal{E}_{ij}^{-1}\|_C}\ge \nu|A_{ij} - A_{i,J} A_{I,J}^{-1} A_{I,j}|.
\end{equation}
Then carefully examining the steps of the proof above yields the desired result of 
Theorem~\ref{mjlai06082021}.  
\end{proof}

Without loss of generality, we may assume matrix $A$ is normalized such that its largest singular value equals to $1$. Let us now give a few simple estimates in the following. 

\begin{corollary}
    Assume $\sigma_1(A) =1$ and consider $r=1$. Then
    \begin{equation*}
        \|A-A_r\|_C\leq \frac{2\sigma_2(A)}{\sqrt{1+\sigma_2^2(A)}}.
    \end{equation*}
    Consequently, when $\sigma_1(A)=1$ and $\sigma_2(A)=\frac{1}{2}$, the estimate is 
    \begin{equation*}
        \|A-A_r\|_C\leq \frac{2\sigma_2(A)}{\sqrt{1+\sigma_2^2(A)}}\le \frac{2\sqrt{5}}{5}.
    \end{equation*}
\end{corollary}

\begin{corollary}
    Using the decreasing property of the singular values, we have
    \begin{equation}
   \|A-A_r\|_C\le    \frac{r+1}{\sqrt{1+\frac{r\sigma^2_{r+1}(A)}{(r+1)\|A\|^2_2}}}\sigma_{r+1}(A).
    \end{equation}
\end{corollary}
\begin{proof}
By Lemma~\ref{C=AB}, we get
\begin{eqnarray}
1 &\le& \sigma_{r+1} (\mathcal{E}_{ij}) \|\mathcal{E}_{ij}^{-1}\|_2 
=  \sigma_{r+1} (\mathcal{E}_{ij}) \|\mathcal{E}_{ij}^{-1}\|_F 
\sqrt{1- \frac{\sum_{k=2}^{r+1}
\sigma_k^2(\mathcal{E}_{ij}^{-1})}{\sum_{k=1}^{r+1}\sigma_k^2(\mathcal{E}_{ij}^{-1})}} 
\cr
&\le& \sigma_{r+1}(A) (r+1)\|\mathcal{E}_{ij}^{-1}\|_C\sqrt{1- \frac{r
\sigma_{r+1}^2(\mathcal{E}_{ij}^{-1})}{(r+1)\sigma_1^2(\mathcal{E}_{ij}^{-1})}}.
\end{eqnarray}
Note that $\sigma_1(\mathcal{E}_{ij}^{-1})=1/\sigma_{r+1}(\mathcal{E}_{ij})$ and $\sigma_{r+1}(\mathcal{E}_{ij}^{-1})=1/\sigma_{1}(\mathcal{E}_{ij})$,
so we have
$$
1- \frac{r
\sigma_{r+1}^2(\mathcal{E}_{ij}^{-1})}{(r+1)\sigma_1^2(\mathcal{E}_{ij}^{-1})}
= 1 - \frac{r \sigma^2_{r+1}(\mathcal{E}_{ij})}{(r+1)\sigma^2_1(\mathcal{E}_{ij})}
\le 1 - \frac{r \sigma^2_{r+1}(\mathcal{E}_{ij})}{(r+1)\|A\|^2}.
$$
From (\ref{ptest}), we can choose
the pair of index $(i,j)$ 
such that $\|A-A_r\|_C=\|E\|_C \le \sigma_{r+1}(\mathcal{E}_{ij}) (r+1)$ and hence, 
\begin{eqnarray*}
\|E\|_C^2 &\le& (r+1)^2 \sigma_{r+1}^2(A) \left(
 1 - \frac{r \sigma^2_{r+1}(\mathcal{E}_{ij})}{(r+1)\|A\|^2}\right) \cr
 &\le& (r+1)^2 \sigma_{r+1}^2(A) \left(1 - \frac{r \|E\|^2_C }{(r+1)^3\|A\|_2}\right).
 \end{eqnarray*}
 Solving $\|E\|_C^2$ from the above inequality, we have
\begin{eqnarray}
\|A-A_r\|_C=\|E\|_C \le \frac{r+1}{\sqrt{1+ \frac{r \sigma^2_{r+1}(A)}{(r+1)\|A\|^2_2}}} 
\sigma_{r+1}(A).
\end{eqnarray}

\end{proof}

Finally, we present two interesting cases when the singular values are decaying at polynomial rate and exponential rate.

\begin{corollary}
Suppose that $\sigma_k(A)\le 1/k$ for all $k\ge 1$. It follows that 
\begin{equation}
        \|A-A_r\|_C\leq \frac{r+1}{\sqrt{(r+1)^3/3}}=  \frac{\sqrt{3}}{\sqrt{r+1}}.
    \end{equation}
In the other words, we need to choose $r\ge 3$ in order for the approximation error in Chebyshev norm to be less than $1$.    
\end{corollary}

\begin{corollary}
Suppose that the singular values of $A$ are
exponentially decreasing, i.e 
$\sigma_k(A)\le C\alpha^{-k}$ for a positive constant $C$ and $\alpha>1$. Then it follows that
\begin{equation}
   \|A-A_r\|_C\le  \frac{C(r+1)}{\alpha^{r+1}}.
\end{equation}
\end{corollary}

The proofs of these above results are straightforward, we leave them to the interested readers.

\section{Greedy Maximum Volume Algorithms} \label{secAlg}
By Theorem~\ref{mjlai06082021}, we can use a dominant submatrix for cross approximation almost as well as the maximal volume submatrix. The aim of this section is to find a dominant submatrix which is close to the submatrix with maximal volume efficiently.  We shall introduce a greedy step to improve the performance of standard maximal volume algorithm introduced in  \cite{GOSTZ08}. 
We start with the maxvol algorithm described in \cite{GOSTZ08} 
to find an approximation for a dominant $r \times r$ submatrix of an $n \times r$ matrix $M$. A detailed discussion of greedy maximum volume algorithms can also be found in \cite{A19}.


\begin{algorithm}[H]
\caption{Maximal Volume Algorithm\label{maxvolalg}}
\begin{algorithmic}[1]
\REQUIRE $n \times r$ matrix $M$, $r\times r$ submatrix $A_0$ with $\det(A_0) \neq 0$, tolerance $\epsilon > 0$
\ENSURE $A_l$ a close to dominant submatrix of $M$
\STATE{Let $l = 0, B_0 = MA_0^{-1}$ and $A_l=A_0$}
\STATE{Set $b_{ij}$ equal to the largest in modulus entry of $B_0$.}
\WHILE{$|b_{ij}| > 1 + \epsilon$}
\STATE{Replace the $j$th row of $A_l$ with the $i$th row of $M$}
\STATE{$l := l+1$}
\STATE{Set $B_l = MA_l^{-1}$}
\STATE{Set $b_{ij}$ equal to the largest in modulus entry of $B_l$}
\ENDWHILE
\end{algorithmic}
\end{algorithm}

The maximal volume algorithm~\ref{maxvolalg} generates a sequence of increasing volumes of submatrices. In other words, we have the following theorem.

\begin{theorem} [cf. \cite{GOSTZ08}]
The sequence $\{v_l\} = \{\vol(A_l)\}$ is increasing and is bounded above. Hence, Algorithm~\ref{maxvolalg}
converges. 
\end{theorem}
\begin{proof}
It is easy to understand that $\vol(A_l)$ is increasing.  
Since the determinant of all submatrices of $A$ are bounded above by Hadamard's inequality, the sequence 
is convergent and hence
the algorithm will terminate and produce a good approximation of dominant submatrix.
\end{proof} 

We now  generalize Algorithm~\ref{maxvolalg} 
to find a $r \times r$ dominant submatrix of an $n \times m$ matrix $M$ 
by searching for the 
largest in modulus entry of both $B_l = M(:,J_l)A_l^{-1}$, 
and $C_l = A_l^{-1}M(I_l,:)$ at each step.


\begin{algorithm}[H]
\caption{2D Maximal Volume Algorithm\label{2dmvalg}}
\begin{algorithmic}[1]
\REQUIRE $n \times m$ matrix $M$, $r\times r$ submatrix 
$A_0$ with $\det(A_0) \neq 0$, tolerance 
$\epsilon > 0$, $l=0$, $b_{ij} = \infty$ and $A_l=A_0$
\ENSURE $A_l$ a close to dominant submatrix with indices $(I_l,J_l)$ in $M$
\WHILE{$|b_{ij}| > 1 + \epsilon$}
\STATE{Let $B_l = M(:,J_l)A_l^{-1}$, and $C_l = A_l^{-1}M(I_l,:)$}
\STATE{Set $b_{ij}$ equal to the largest in modulus entry of both $B_l$ and $C_l$}
\IF{$b_{ij}$ is from $B_l$}
\STATE{Replace the $j$th row of $A_l$ with the $i$th row of $M(:,J_l)$}
\ENDIF
\IF{If $b_{ij}$ is from $C_l$}
\STATE{Replace the $i$th column of $A_l$ with the $j$th column of $M(I_l,:)$}
\ENDIF
\STATE{  $l := l+1$}
\ENDWHILE
\end{algorithmic}
\end{algorithm}

However, Algorithm~\ref{2dmvalg} 
requires two backslash operations at each step. 
To simplify this to one backslash operation at each step, 
we may consider an alternating maxvol algorithm such as Algorithm \ref{alg3} where we alternate 
between optimizing swapping rows and columns. Note that this 
converges to a dominant submatrix because the sequence $\vol(A_l)$ is again increasing 
and bounded above.

\begin{algorithm}[H]
 \caption{Alternating Maximal Volume Algorithm \label{alg3}}
\begin{algorithmic}[1]
\REQUIRE $n \times m$ matrix $M$, $r\times r$ submatrix $A_0$ with $\det(A_0) \neq 0$, tolerance $\epsilon > 0$, 
$l=0$, $b_{ij} = \infty$ and $A_l=A_0$
\ENSURE $A_l$ a close to dominant submatrix of $M$ with index set $(I_l,J_l)$ in 
$M$.
\WHILE{$\max\{|b_{ij}|,|c_{ij}|\} > 1 + \epsilon$}
\STATE{Let $B_l = M(:,J_l)A_l^{-1}$.}
\STATE{Set $b_{ij}$ equal to the largest in modulus entry of $B_l$.}
\IF{$|b_{ij}| > 1 + \epsilon$}
\STATE{Replace the $j$th row of $A_l$ with the $i$th row 
of $M(:,J_l)$.}
\ENDIF
\STATE{Let $C_l = A_l^{-1}M(I_l,:)$.}
\STATE{Set $c_{ij}$ equal to the largest in modulus entry of $C_l$.}
\IF{$|c_{ij}| > 1 + \epsilon$}
\STATE{Replace the $i$th column of $A_l$ with the $j$th column of $M(I_l,:)$.}
\ENDIF
\STATE{$l := l+1$.}
\ENDWHILE
\end{algorithmic}
\end{algorithm}

We now introduce a greedy strategy to reduce the number of backslash operations 
needed to find a dominant submatrix by swapping more rows at each step, which 
will be called a greedy maxvol algorithm. 
The greedy maxvol algorithm is similar to the maxvol algorithm except that 
instead of swapping one row every iteration, we may swap two or more rows. 
First, we will describe the algorithm for swapping at 
most two rows of an $n \times r$ matrix, which we will call $2$-greedy maxvol.
Given an $n \times r$ matrix $M$, initial $r \times r$ submatrix $A_0$ such that $\det(A_0) \neq 0$, 
and tolerance $\epsilon > 0$,  we use Algorithm~\ref{gmvalg} to find a dominant matrix.   


\begin{algorithm}[H]
 \caption{$2$-Greedy Maximal Volume Algorithm\label{gmvalg}}
\begin{algorithmic}[1]
\REQUIRE $n \times r$ matrix $M$, $r\times r$ submatrix $A_0$ with $\det(A_0) \neq 0$, tolerance $\epsilon > 0$
\ENSURE $A_l$ a close to dominant submatrix of $M$
\STATE{Let $l = 0, B_0 = MA_0^{-1}$ and $A_l=A_0$.}
\STATE{Set $b_{i_1j_1}$ equal to the largest in modulus entry of $B_0$.}
\WHILE{$|b_{i_1j_1}| > 1 + \epsilon$}
\STATE{Replace the $j_1$th row of $A_l$ with the $i_1$th row of $M$.}
\STATE{Set $b_{i_2j_2}$ equal to the largest in modulus entry of $B_l$ over all rows excluding the 
$j_1$st row.}
\STATE{Let $B'_l = \begin{bmatrix}
    b_{i_1j_1} & b_{i_1j_2} \\
    b_{i_2j_1} & b_{i_2j_2}
    \end{bmatrix}$.}
\IF{$|\det(B'_l)| > |b_{i_1j_1}|$}
\STATE{Replace the $j_2$th row of $A_l$ with the $i_2$th row of $M$}
\ENDIF
\STATE{$l := l+1$.}
\STATE{Let $B_l = MA_l^{-1}$.}
\STATE{Set $b_{i_1j_1}$ equal to the largest in modulus entry of $B_l$.}
\ENDWHILE
\end{algorithmic}
\end{algorithm}

To prove that this algorithm~\ref{gmvalg} converges, 
recall Hadamard's inequality. For an $n \times n$ matrix $M$, we have:
\begin{equation}
\label{hadamard}
|\det(M)| \leq \prod_{i=1}^n ||M(:,i)||.
\end{equation}
With the estimate above, it follows 
\begin{theorem}
The sequence $\{|\det(A_\ell)|, \ell\ge 1\}$ from Algorithm~\ref{gmvalg} is increasing and is bounded above 
by $\prod_{i=1}^n ||M(:,i)||$. Therefore, Algorithm~\ref{gmvalg} converges.
\end{theorem}
\begin{proof}
Note that multiplication by an invertible matrix does not change the ratio of determinants of 
pairs of corresponding sub-matrices. Suppose we only swap one row. Then $\frac{\det(A_{l+1})}{\det(A_l)} = 
\frac{b_{i_1j_1}}{\det(I)} = b_{i_1j_1}$. Therefore, as $|b_{i_1j_1}| > 1$, we have$|\det(A_{l+1})| > 
|\det(A_l)|$.

Now suppose we swap two rows. Then after permutation of rows, the submatrix of $B_n$ corresponding to $A_{l+1}$
 is $\begin{bmatrix}
B'_l & R \\
0 & I
\end{bmatrix}$ in block form, which has determinant $\det(B'_l)$. Therefore similarly to before, we have that 
$\det(A_{l+1}) = \det(B'_l) \det(A_l)$, and since $|\det(B'_l)| > |b_{i_1j_1}| > 1$, we have that 
$|\det(A_{l+1})| > |\det(A_l)|$, so $|\det(A_l)|$ is an increasing sequence and bounded above.
\end{proof}

Note that when we swap two rows, the ratio between $|\det(A_{l+1})|$ and $|\det(A_l)|$ is maximized when 
$|\det(B'_l)|$ is maximized. 
For a more general algorithm, we search for a largest element in each of the $r$ rows of $B_l$: $|b_{i_1j_1}| 
\geq |b_{i_2j_2}| 
\geq \dots \geq |b_{i_rj_r}|$. We'll call $b_k := b_{i_kj_k}$. Let $B_l^{(k)} = \begin{bmatrix}
b_{i_1j_1} & b_{i_1j_2} & \dots & b_{i_1j_k} \\
b_{i_2j_1} & b_{i_2j_2} & \dots & b_{i_2j_k} \\
\vdots & \vdots & & \vdots \\
b_{i_kj_1} & b_{i_kj_2}  & \dots & b_{i_kj_k}
\end{bmatrix}$. Then we replace the $j_k$th row of $A_n$ with the $i_k$th row of $M$ if $|\det(\begin{bmatrix}
B_l^{(k)} & x_k \\
y_k & b_{k+1}
\end{bmatrix})| \geq |\det(
B_l^{(k)})|$, where $x_k = \begin{bmatrix}
b_{i_1j_k} \\  b_{i_2j_k} \\ \vdots \\ b_{i_{k-1}j_k}
\end{bmatrix}$, and $y_k = \begin{bmatrix} b_{i_kj_1} & b_{i_kj_2}  & \dots & b_{i_kj_{k-1}} \end{bmatrix}$. 
Using the Schur determinant formula for the determinant of block-matrices, 
this condition is equivalent to the condition that $|b_{k+1} - 
d_k[B_l^{(k)}]^{-1}c_k| \geq 1$. In particular, if $|b_{k+1}| > 1$, and $\hbox{sign}(b_{k+1}) \neq 
\hbox{sign}(d_k[B_l^{(k)}]^{-1}c_k)$, then the condition will be met. The general $h$-greedy maxvol algorithm for swapping at most $h$ rows at each 
iteration runs can be derived and the detail is left.

We may generalize this algorithm to $n \times m$ matrices by employing an alternating $h$-Greedy algorithm similar to before 
between the rows and columns. 
We have the following alternating $h$-greedy maxvol algorithm summarized in Algorithm \ref{hGreedy}. For convenience, we shall call greedy maximal volume algorithms (GMVA) for these algorithms discussed in this section. Let us also introduce
two matrices 
\begin{equation}
\label{B}
B'_k = \begin{bmatrix}
    b_{i_1j_1} & b_{i_1j_2} & \dots & b_{i_1j_k} \\
    b_{i_2j_1} & b_{i_2j_2} & \dots & b_{i_2j_k} \\
    \vdots & \vdots & & \vdots \\
    b_{i_kj_1} & b_{i_kj_2}  & \dots & b_{i_kj_k}
    \end{bmatrix}
\end{equation}
and
\begin{equation}
 \label{C}
C'_k = \begin{bmatrix}
    c_{i_1j_1} & c_{i_1j_2} & \dots & c_{i_1j_k} \\
    c_{i_2j_1} & c_{i_2j_2} & \dots & c_{i_2j_k} \\
    \vdots & \vdots & & \vdots \\
    c_{i_kj_1} & c_{i_kj_2}  & \dots & c_{i_kj_k}
    \end{bmatrix}
\end{equation}
to simplify the notations.

\begin{algorithm}[H] 
\caption{Alternating $h$-Greedy Maximal Volume Algorithm \label{hGreedy}}
{\small
\begin{algorithmic}[1]
\REQUIRE $n \times m$ matrix $M$, $r\times r$ submatrix $A_1$ with index set $(I_0,J_0)$ with $\det(A_1) \neq 0$, tolerance $\epsilon > 0$, $l = 0$, and $b_{i_1j_1}=c_{i_1j_1}=\infty$
\ENSURE $A_l$ a close to dominant submatrix of $M$ with index set $(I_l,J_l)$
\WHILE{$\max\{|b_{i_1j_1}|,|c_{i_1j_1}|\} > 1 + \epsilon$}
\STATE{Let $B_l = M(:,J_l)A_l^{-1}$.} 
\STATE{Set $b_{i_1j_1}$ equal to the largest in modulus entry of $B_l$.}
\IF{$|b_{i_1j_1}| > 1 + \epsilon$}
\STATE{Replace the $j_1$th row of $A_l$ with the $i_1$th row of $M$}
\ENDIF
\FOR{$k = 2:h$}
\STATE{Set $b_{i_k j_k}$ be the largest in modulus entry of $B_l$ over all row excluding the 
$j_1, \dots , j_{k-1}$ rows.}
\STATE{Let $B_k'$ be the matrix in (\ref{B})}
\IF{$|\det(B'_k)| > |\det(B'_{k-1})|$}
\STATE{Replace the $j_k$th row of $A_l$ with the $i_k$th row of $M$}
\STATE{\textbf{break}}
\ENDIF
\ENDFOR
\STATE{Let $C_l = A_l^{-1}M(I_l,:)$.}
\STATE{Set $c_{i_1j_1}$ equal to the largest in modulus entry of $C_l$.} 
\IF{$|c_{i_1j_1}| > 1 + \epsilon$}
\STATE{Replace the $i_1$th column of $A_l$ with the $j_1$th column of $M$}
\ENDIF
\FOR{$k = 2:h$}
\STATE{Set $c_{i_kj_k}$ be the largest in modulus entry of $C_l$ over all columns excluding the 
$i_1, \dots , i_{k-1}$ columns.}
\STATE{Let $C_k'$ be the matrix in (\ref{C})}
\IF{$|\det(C'_k)| > |\det(C'_{k-1})|$}
\STATE{Replace the $i_k$th column of $A_l$ with the $j_k$th column of $M$}
\STATE{\textbf{break}}
\ENDIF
\ENDFOR
\STATE{$l := l+1$}
\ENDWHILE
\end{algorithmic}
}
\end{algorithm}

\section{Numerical Experiments on GMVA and Applications} \label{secNum}
In this section, we first compare the numerical performances of GMVA by varying $h$ to demonstrate that our GMVA is much better than the maxvol algorithm (the case $h=1$).
Then we present two common  applications of the matrix cross approximation. That is, we apply the matrix cross approximation to image compression and least squares approximation of continuous functions via standard polynomial basis.  In addition, GMVA has also been applied to the compression of plasma physics data \cite{de2023compression} and high dimensional function approximation in \cite{LaiShenKST1,LaiShenKST2,shen_sparse}.

\subsection{Performance of GMVA}
First we report some simulation results to answer the questions: how
many backslash operations, and how much computational time the $h$-greedy maxvol algorithm can save when using the greedy maxvol algorithm. It is worth mentioning that to the authors in \cite{GOSTZ08} point out there is a fast way of calculating the matrix inverse in Algorithm~\ref{maxvolalg}, however, we skipped the implementation of those fast steps as they require a more prudent choice of the initial guess matrix.

We will generate $100$ random $5000\times r$ matrices and $5000\times 5000$ matrices, and calculate the average number of backslash operations, and average computational time needed to find a dominant submatrix within a relative error of $10^{-8}$ for $r=30,60,90,120,150,180,210,240$. We choose random row and column indices for generating the initial $r\times r$ matrix when conducting the experiments across all $h$.
The results are presented in Table~\ref{av_backslash1} and Table~\ref{av_backslash2}. As we can see,
the computational time and number of backslash operations needed decreases as $h$ increases up to $r$. Note that $h=1$ corresponds to the original maximal volume algorithm. This shows that our GMVA is more efficient when $h$ is large.

\begin{table}[t]
\centering
\caption{Average number of backslash operations and computational time (in seconds) needed to find a dominant submatrix of $100$ random $5000 \times r$ matrices using $h$-greedy maxvol algorithm within a relative error of $10^{-8}$. \label{av_backslash1}}
\vspace{3mm}
\begin{tabular}{l|llllc}
\toprule
  & \multicolumn{5}{c}{Average Number of Backslash Operations / Computational Time}                                                                         \\ 
\cmidrule{2-6}
   $r$  & \multicolumn{1}{c|}{$h = 1$} & \multicolumn{1}{c|}{$h = 2$} & \multicolumn{1}{c|}{$h = 3$} & \multicolumn{1}{c|}{$h = 4$} & $h = r$ \\ \midrule
$30$  & \multicolumn{1}{l|}{33.92 / 0.131}   & \multicolumn{1}{l|}{23.84 / 0.120}   & \multicolumn{1}{l|}{20.88 / 0.107}   & \multicolumn{1}{l|}{20.33 / 0.105}   & 19.84 / 0.103  \\
$60$  & \multicolumn{1}{l|}{50.56 / 0.265}   & \multicolumn{1}{l|}{34.05 / 0.212}   & \multicolumn{1}{l|}{31.78 / 0.199}   & \multicolumn{1}{l|}{31.31 / 0.197}   & 29.56 / 0.188   \\
$90$  & \multicolumn{1}{l|}{62.68 / 0.513}   & \multicolumn{1}{l|}{42.91 / 0.422}   & \multicolumn{1}{l|}{38.39 / 0.382}   & \multicolumn{1}{l|}{37.29 / 0.372}   & 38.06 / 0.384   \\
$120$ & \multicolumn{1}{l|}{71.64 / 0.864}   & \multicolumn{1}{l|}{51.46 / 0.717}   & \multicolumn{1}{l|}{44.43 / 0.619}   & \multicolumn{1}{l|}{42.57 / 0.597}   & 41.12 / 0.585   \\
$150$ & \multicolumn{1}{l|}{81.97 / 1.289}   & \multicolumn{1}{l|}{54.78 / 1.019}   & \multicolumn{1}{l|}{50.08 / 0.944}   & \multicolumn{1}{l|}{46.02 / 0.872}   & 46.33 / 0.895   \\
$180$ & \multicolumn{1}{l|}{89.75 / 1.853}   & \multicolumn{1}{l|}{58.93 / 1.404}   & \multicolumn{1}{l|}{54.06 / 1.286}   & \multicolumn{1}{l|}{52.54 / 1.258}   & 53.10 / 1.298   \\
$210$ & \multicolumn{1}{l|}{95.68 / 2.473}   & \multicolumn{1}{l|}{64.82 / 1.929}   & \multicolumn{1}{l|}{57.93 / 1.741}   & \multicolumn{1}{l|}{55.25 / 1.664}   & 53.37 / 1.629   \\
$240$ & \multicolumn{1}{l|}{99.65 / 3.022}   & \multicolumn{1}{l|}{69.85 / 2.434}   & \multicolumn{1}{l|}{60.77 / 2.135}   & \multicolumn{1}{l|}{57.39 / 2.025}   & 55.55 / 2.012  \\
\bottomrule
\end{tabular}
\end{table}

\begin{table}[t]
\centering
\caption{Average number of backslash operations and computational time (in seconds) needed to find a dominant submatrix of $100$ random $5000 \times 5000$ matrices using $h$-greedy maxvol algorithm within a relative error of $10^{-8}$. \label{av_backslash2}}
\vspace{3mm}
\begin{tabular}{l|llllc}
\toprule
   & \multicolumn{5}{c}{Average Number of Backslash Operations / Computational Time}    \\ \cmidrule{2-6} 
  $r$    & \multicolumn{1}{c|}{$h = 1$} & \multicolumn{1}{c|}{$h = 2$} & \multicolumn{1}{c|}{$h = 3$} & \multicolumn{1}{c|}{$h = 4$} & $h = r$ \\ \midrule
$30$  & \multicolumn{1}{l|}{33.63 / 0.136}   & \multicolumn{1}{l|}{23.58 / 0.093}   & \multicolumn{1}{l|}{21.16 / 0.080}   & \multicolumn{1}{l|}{20.49 / 0.075}   & 19.98 / 0.073   \\
$60$  & \multicolumn{1}{l|}{51.22 / 0.417}   & \multicolumn{1}{l|}{34.65 / 0.264}   & \multicolumn{1}{l|}{31.81 / 0.246}   & \multicolumn{1}{l|}{30.39 / 0.235}   & 30.13 / 0.237   \\
$90$  & \multicolumn{1}{l|}{65.52 / 0.835}   & \multicolumn{1}{l|}{44.31 / 0.558}   & \multicolumn{1}{l|}{38.75 / 0.492}   & \multicolumn{1}{l|}{38.71 / 0.494}   & 37.06 / 0.480  \\
$120$ & \multicolumn{1}{l|}{74.13 / 1.353}   & \multicolumn{1}{l|}{48.49 / 0.889}   & \multicolumn{1}{l|}{45.42 / 0.839}   & \multicolumn{1}{l|}{43.79 / 0.811}   & 43.09 / 0.813   \\
$150$ & \multicolumn{1}{l|}{80.81 / 2.007}   & \multicolumn{1}{l|}{55.44 / 1.387}   & \multicolumn{1}{l|}{51.24 / 1.295}   & \multicolumn{1}{l|}{47.77 / 1.211}   & 48.24 / 1.244   \\
$180$ & \multicolumn{1}{l|}{89.85 / 2.771}   & \multicolumn{1}{l|}{60.85 / 1.893}   & \multicolumn{1}{l|}{52.64 / 1.652}   & \multicolumn{1}{l|}{50.52 / 1.588}   & 49.04 / 1.568   \\
$210$ & \multicolumn{1}{l|}{95.74 / 3.667}   & \multicolumn{1}{l|}{63.96 / 2.475}   & \multicolumn{1}{l|}{57.88 / 2.254}   & \multicolumn{1}{l|}{56.89 / 2.229}   & 53.16 / 2.106   \\
$240$ & \multicolumn{1}{l|}{100.45 / 4.576}  & \multicolumn{1}{l|}{67.64 / 3.107}   & \multicolumn{1}{l|}{59.38 / 2.755}   & \multicolumn{1}{l|}{59.04 / 2.746}   & 55.41 / 2.622  \\
\bottomrule
\end{tabular}
\end{table}

\subsection{Application to image compression}
Many images are of low rank in nature or can be approximated by a low rank image. More precisely, the pixel values of an image in RGB colors are in integer matrix format.   
Using standard SVD to approximate a given image will lead to a low rank approximation with non-integer entries, which 
require more bits to store and transmit and hence defeat the purpose of compression of an image. 

In this subsection, we introduce the cross approximation for image compression. Let $M$ be a matrix with integer entries, we apply maximal volume algorithm to find a $r\times r$ submatrix $A_r$ and index sets $I$ and $J$ for row and column indices.  Then we can compute a
cross approximation $M(:,J) A_r^{-1} M(I,:)$ to approximate $M$. Let $U, S, V$ be the SVD approximation of $M$ with $S$ being the matrix of singular values, if $S(r+1,r+1)$ is small enough, by Theorem~\ref{mjlai06062021}, 
the cross approximation is a good one. One only needs to store and transmit $2nr-r^2$ integer entries for this image as both $M(:,J)$ and $M(I.:)$ contain the block $A_r$. Note that we do not need to save/store/transmit $A_r^{-1}$. Instead, it will be computed from $A_r$ after the storage/transmission. 

We shall present a few examples to demonstrate the effectiveness of this computational method. 
In practice, we only need to measure the peak signal-to-noise ratio (PSNR) to determine if $r$ is large enough or not. That is, do many decomposition and reconstruction to find the best rank $r$ as well as index pair $(I,J)$ before  storing or sending the components $M(:,J), M(I,J^c)$ 
which are compressed by using the standard JPEG, where $J^c$ stands for the complement index of $J$. Since the size of the two components is smaller than the size of the original entire image, the application of JPEG to the two components should yield a better compression performance than the original image and hence, the matrix cross approximation can help further increase the performance of the image compression.  

In the following, We use 13 standard images to test our method and 
report our results in Table~\ref{caex1}. That is, the PSNR for these images with 
their ranks are shown in Table~\ref{caex1}. In addition, in 
Figure~\ref{caImage1}, we show the comparison of the original images with 
reconstructed images
from their matrix cross approximation. Table~\ref{caex1} shows that 
many images can be recovered to have PSNR equals to $32$ or better by using fewer than 
half of the integer entries of the entire matrix. In other words, about the 
half of the pixel values are needed to be encoded for compression 
instead of the entire image. 
 
\begin{table}[h]
\caption{Various image compression results \label{caex1}}
\centering
\vspace{3mm}
\begin{tabular} {l|c|c|r|r|r }
\toprule
 images & m & n  & rank &Ratio & PSNR\\ \midrule
bank   &    512  &    512 &     355 &   0.90 &    32.12\\
barbara &    512  &    512 &     260 &   0.76 &    32.22\\
clock  &    256  &    256 &     90 &   0.58 &    32.48\\
lena &    512  &    512 &     240 &   0.72 &    32.20\\
house  &    256  &    256 &     90 &   0.58 &    32.38\\
myboat &    256  &    256 &     135 &   0.78 &    32.80\\
saturn &    512  &    512 &      55 &   0.20 &    32.21\\
F16    &    512  &    512 &     235 &   0.71 &    32.21\\
pepper &    512  &    512 &     370 &   0.92 &    32.23\\
knee   &    512  &    512 &     105 &   0.37 &    32.24\\
brain  &    512  &    512 &     110 &   0.38 &    32.62\\
mri   &    512  &    512 &     130 &   0.44 &    34.46\\
\bottomrule
\end{tabular}
\end{table}

\begin{figure}[htpb]
\centering
\begin{tabular}{c} 
\includegraphics[width=0.8\textwidth,
height=0.4\textwidth]{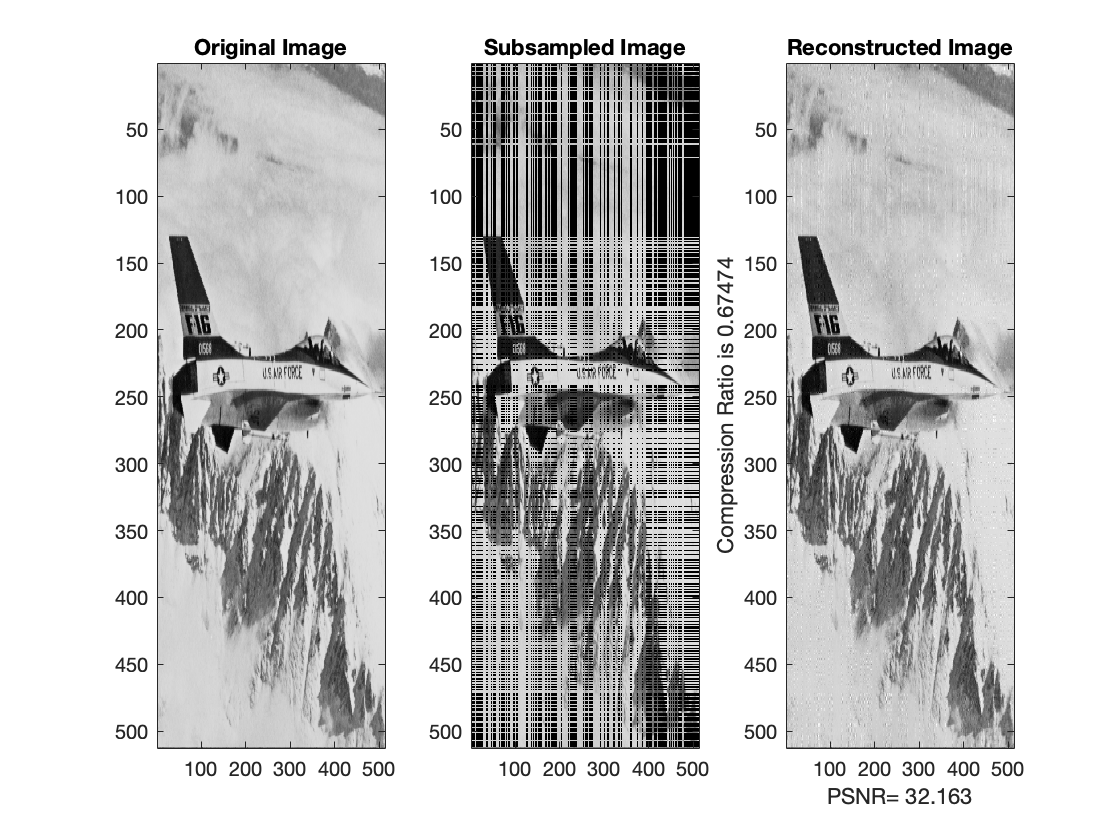}  \\
\includegraphics[width=0.8\textwidth,
height=0.4\textwidth]{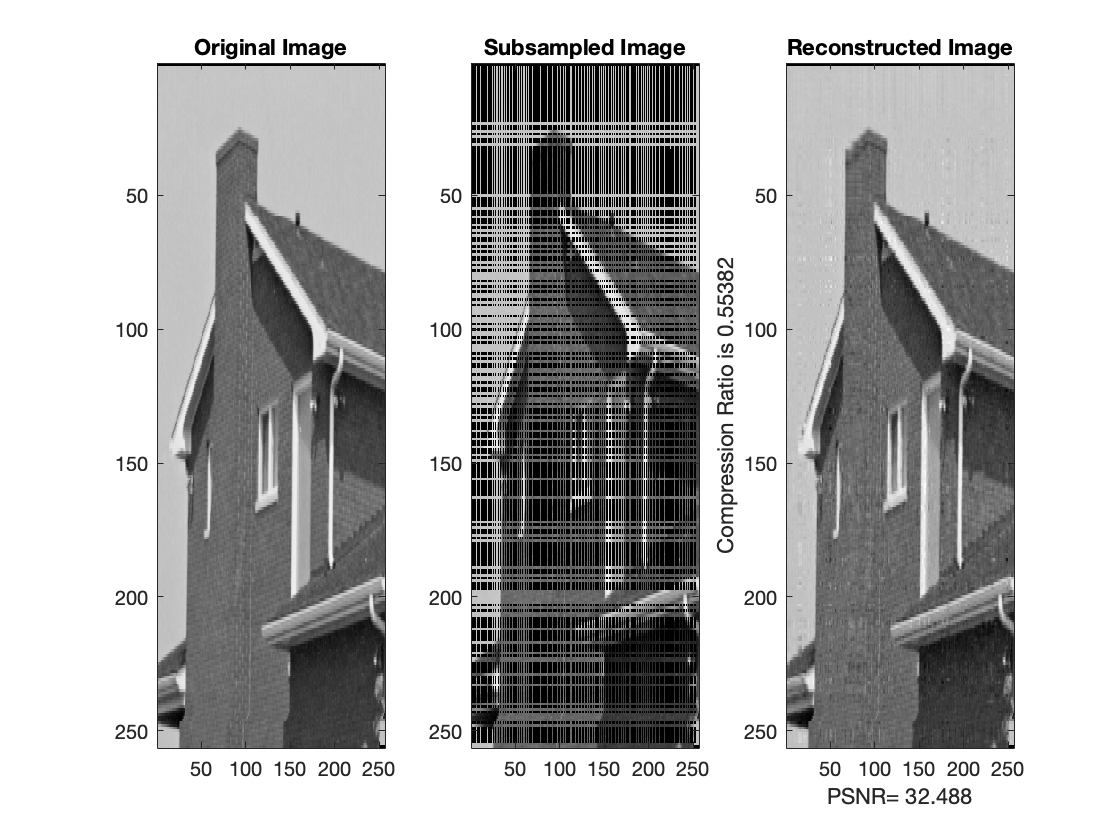} \\
\includegraphics[width=0.8\textwidth,
height=0.4\textwidth]{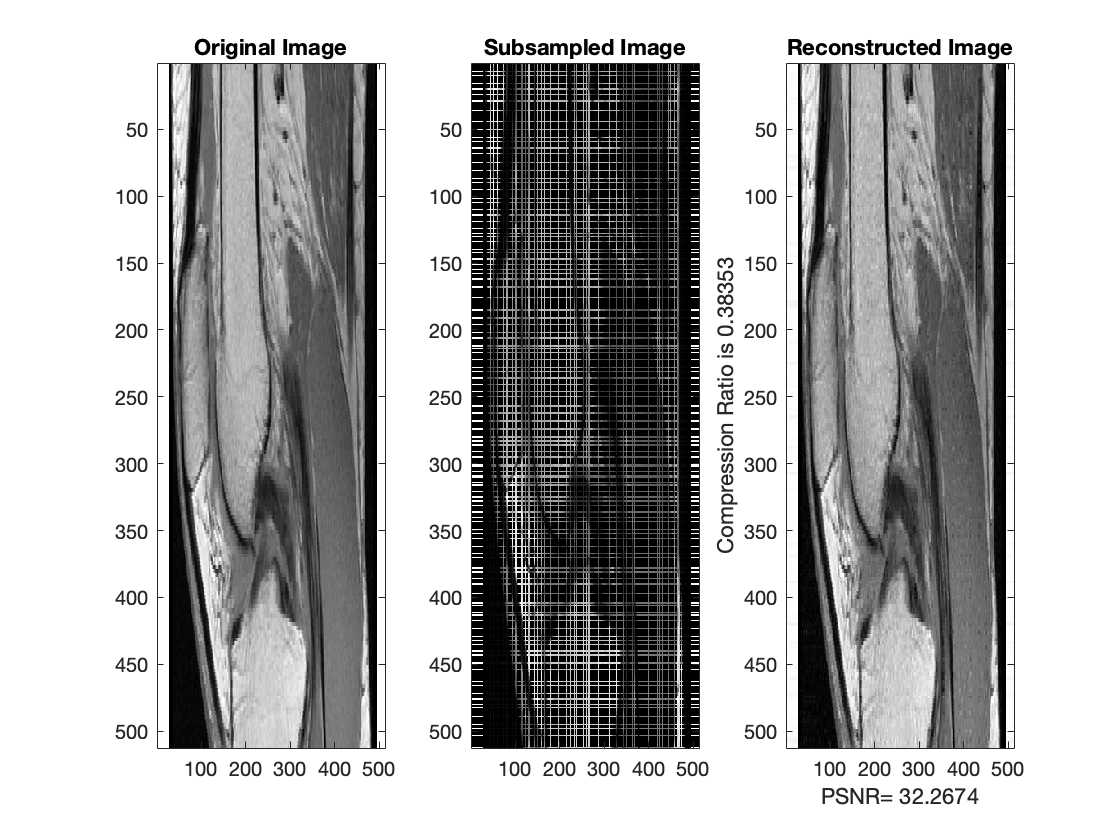} \\
\end{tabular}
\caption{Comparison of original images, subsampled images, and reconstructed images \label{caImage1}}
\end{figure}

\subsection{Application to least squares approximation}
Next, we turn our attention to the approximation of least squares problem 
$\min_{\bf x} \|A \bfx -\bfb\|$ for a given observation matrix $A$ and observation vector $\bfb$. 
We shall write $\bfx_{\bfb}$ to the least
square solution and explain how to use the matrix cross approximation to find a
good approximation of $\bfx_{\bfb}$.  As explained before, we shall use $A_{:,J}A_{I,J}^{-1}A_{I,:}$ to approximate the observation matrix $A$. For convenience, let us write 
\begin{equation}
    A = 
    \begin{bmatrix}
    A_{I,J} & A_{I,J^c} \\
    A_{I^c, J} & A_{I^c, J^c}
\end{bmatrix}.
\end{equation}
The matrix cross approximation is 
\begin{equation}
A_{:,J}A_{I,J}^{-1}A_{I,:}=
    \begin{bmatrix}
        A_{I,J} \\
        A_{I^c,J}
    \end{bmatrix} A_{I,J}^{-1} 
    \begin{bmatrix}
        A_{I,J} & A_{I,J^c}
    \end{bmatrix} 
    =
    \begin{bmatrix}
        E_I \\
        A_{I^c,J}A_{I,J}^{-1}
    \end{bmatrix}
    \begin{bmatrix}
        A_{I,J} & A_{I,J^c}
    \end{bmatrix},
\end{equation}
where $E_I$ is the identity matrix. For simplicity, we use Theorem~\ref{mjlai06062021} to have
\begin{equation} \label{lsqInequ1}
    \|A_{I^c,J^c}-A_{I^c,J}A_{I,J}^{-1}A_{I,J^c}\|_C\leq \frac{(r+1)\sigma_{r+1}(A)}{\sqrt{1+\sum_{k=1}^{r}\frac{\sigma_{r+1}^2(A)}{\sigma_k^2(A)}}}.
\end{equation}
Note that all the entries in $A_{I^c,J}A_{I,J}^{-1}$ are less than or equal to 1 in absolute value. One way to find an approximation of the least squares solution is to solve
\begin{equation} \label{upperhalfLSQ}
    \begin{bmatrix}
        A_{I,J} & A_{I,J^c}
    \end{bmatrix}
    \mathbf{\hat{x}}
    = \mathbf{b}_I
\end{equation}
where $\mathbf{b}=[\mathbf{b}_I;\mathbf{b}_{I^c}]$ consists of two parts according to the indices in $I$ and the complement indices in $I^c$. Such a solution can be solved using the sparse solution techniques as discussed in \cite{LW2021,shen_sparse} or simply solve $A_{I,J}\mathbf{\hat{x}}=\mathbf{b}_I$ with zero entries in $\bfx$ associated with the indices in $J^c$. Hence, there is no residual or near zero residual in the first part. Our main concern is the residual vector in the second part 
\begin{equation}
    \begin{bmatrix}
        A_{I^c,J} & A_{I^c,J^c} 
    \end{bmatrix}
    \mathbf{\hat{x}} - \mathbf{b}_{I^c}=A_{I^c,J}\mathbf{\hat{x}}-\mathbf{b}_{I^c}=A_{I^c,J}A_{I,J}^{-1}\mathbf{b}_I-\mathbf{b}_{I^c}.
\end{equation}
Recall $\bfx_\bfb$ is a least squares solution. Let us write the residual vector $A\mathbf{x}_\bfb-\mathbf{b}=:\mathbf{\epsilon}$ with $\mathbf{\epsilon}$ consisting of two parts, i.e., $\mathbf{\epsilon}=[\mathbf{\epsilon}_I;\mathbf{\epsilon}_{I^c}]$. Then $\mathbf{b}_{I^c}=
\begin{bmatrix}
    A_{I^c,J} & A_{I^c,J^c}
\end{bmatrix}
\mathbf{x}_b-\epsilon_{I^c}$. We have
\begin{align*}
    A_{I^c,J}\mathbf{\hat{x}}-\mathbf{b}_{I^c}&=A_{I^c,J}A_{I,J}^{-1}\mathbf{b}_I-\mathbf{b}_{I^c} \\
    & = A_{I^c,J}A_{I,J}^{-1}\mathbf{b}_I-\begin{bmatrix}
        A_{I^c,J} & A_{I^c,J^c}
    \end{bmatrix}
    \mathbf{x}_b + \epsilon_{I^c} \\
    & = A_{I^c,J}A_{I,J}^{-1}\mathbf{b}_I - 
    \begin{bmatrix}
        A_{I^c,J} & A_{I^c,J}A^{-1}_{I,J}A_{I,J^c}
    \end{bmatrix}
    \mathbf{x}_b \\
    & +
    \begin{bmatrix}
        0_{I^c,J} & A_{I^c,J}A^{-1}_{I,J}A_{I,J^c}-A_{I^c,J^c}
    \end{bmatrix}\mathbf{x}_b+\epsilon_{I^c} \\
    & = A_{I^c,J}A_{I,J}^{-1}(\mathbf{b}_I-\begin{bmatrix}
        A_{I,J} & A_{I,J^c}
    \end{bmatrix}\mathbf{x}_b) \\
    & +
    \begin{bmatrix}
        0_{I^c,J} & A_{I^c,J}A^{-1}_{I,J}A_{I,J^c}-A_{I^c,J^c}
    \end{bmatrix}\mathbf{x}_b+\epsilon_{I^c} \\
    & = A_{I^c,J}A_{I,J}^{-1}\epsilon_I+
    \begin{bmatrix}
        0_{I^c,J} & A_{I^c,J}A^{-1}_{I,J}A_{I,J^c}-A_{I^c,J^c}
    \end{bmatrix} \mathbf{x}_b+\epsilon_{I^c}.
\end{align*}
It follows that 
\begin{equation}
\begin{aligned}
  \|A\mathbf{\hat{x}}-\mathbf{b}\|_{\infty} & = \| A_{I^c,J}\mathbf{\hat{x}}-\mathbf{b}_{I^c}\|_\infty \cr
 & \leq \|A_{I^c,J}A^{-1}_{I,J}\epsilon_I\|_{\infty}+\|A_{I^c,J}A^{-1}_{I,J}A_{I,J^c}-A_{I^c,J^c}\|_C\|\mathbf{x}_b\|_1+\|\epsilon_{I^c}\|_{\infty} \\
  & \leq \|\epsilon_I\|_1+\|A_{I^c,J}A^{-1}_{I,J}A_{I,J^c}-A_{I^c,J^c}\|_C\|\mathbf{x}_b\|_1+\|\epsilon_{I^c}\|_{\infty}.
\end{aligned}
\end{equation}
since all the entries of $A_{I^c,J}A_{I,J}^{-1}$ are less than or equal to 1 in absolute value,
By applying (\ref{lsqInequ1}), we have established the following main result in this section: 

\begin{theorem}
\label{May2023}
    Let $\mathbf{\epsilon}=[\mathbf{\epsilon}_I;\mathbf{\epsilon}_{I^c}]$ and let the residual vector $A\mathbf{x}_\bfb-\mathbf{b}=\mathbf{\epsilon}$. Then
    \begin{equation} \label{lsqIneq}
        \|A\mathbf{\hat{x}}-\mathbf{b}\|_{\infty}\leq \|\mathbf{\epsilon}_I\|_1+\frac{(r+1)\sigma_{r+1}(A)}{\sqrt{1+\sum_{k=1}^{r}\frac{\sigma_{r+1}^2(A)}{\sigma_k^2(A)}}}\|\mathbf{x}_b\|_1+\|\mathbf{\epsilon}_{I^c}\|_{\infty}.
    \end{equation}
\end{theorem}
The solution $\mathbf{\hat{x}}$ to (\ref{upperhalfLSQ}) is much faster and the residual has a similar  estimate as the original least squares solution 
if $\sigma_{r+1}(A) \ll 1$. The computational time is spent on finding a dominant matrix which is dependent on the number of iterations (MATLAB backslash) of matrices of size $r\times r$. This new computational method will be useful when dealing with least squares tasks which have multiple right-hand side vectors. 

Traditionally, we often use the well-known singular value decomposition (SVD) to 
solve the least squares problem. The computational method explained above has an 
distinct feature that the SVD does not have.  Let us use the following 
numerical example to illustrate the feature.

Now let us present some numerical examples of finding the best least squares approximation of functions of two variables by using a polynomial basis. In general, this scheme can be applied to any dimensional function approximation. 
Suppose that we use polynomials of degrees less or equal to 10, i.e.  
basis functions $1,x,y,x^2, xy, y^2, \cdots,y^{10}$  to find the best coefficients $c_0,c_1,\cdots,c_{65}$ such that 
\begin{equation} 
\label{2DlsqEqn}
   \min_{c_0, \cdots, c_{65}} \| f(x,y) -( c_0+c_1 x+c_2 y  +\cdots+c_{65}y^{10})\|_{L^2}
\end{equation}
for $(x,y)\in [-1,1]^2$.  
We can approximate (\ref{2DlsqEqn}) by using discrete least squares method, i.e., we find the coefficients $\mathbf{c}=(c_0,\cdots,c_n)$ such that
\begin{equation} \label{lsqmin}
    \min_{c_0,\cdots,c_n}\|f(\mathbf{x}_j)-\sum_{i=0}^n c_i\phi_i(\mathbf{x}_j)\|_{\ell^2},
\end{equation}
where $\phi_i$'s are the polynomial basis functions and 
$\mathbf{x}_j$'s are the sampled points in $[-1, 1]^2$ with $n\gg 1$.  

In our experiments, we select continuous functions across different families such as $e^{x^2+y^2}$, $\sin(x^2+y^2)$, $\cos(x^2+y^2)$, $\ln(1+x^2+y^2)$, $(1+x^4+y^4)/(1+x^2+y^2)$, Franke's function, Ackley's function, Rastrigin's function, and wavy function as the testing functions to demonstrate the performances. Some plots of the functions are shown on the left side of Figure~\ref{pivotalpoints}.

\begin{figure}[h]
\centering 
\begin{tabular}{cc}
     \includegraphics[width=0.5\textwidth]{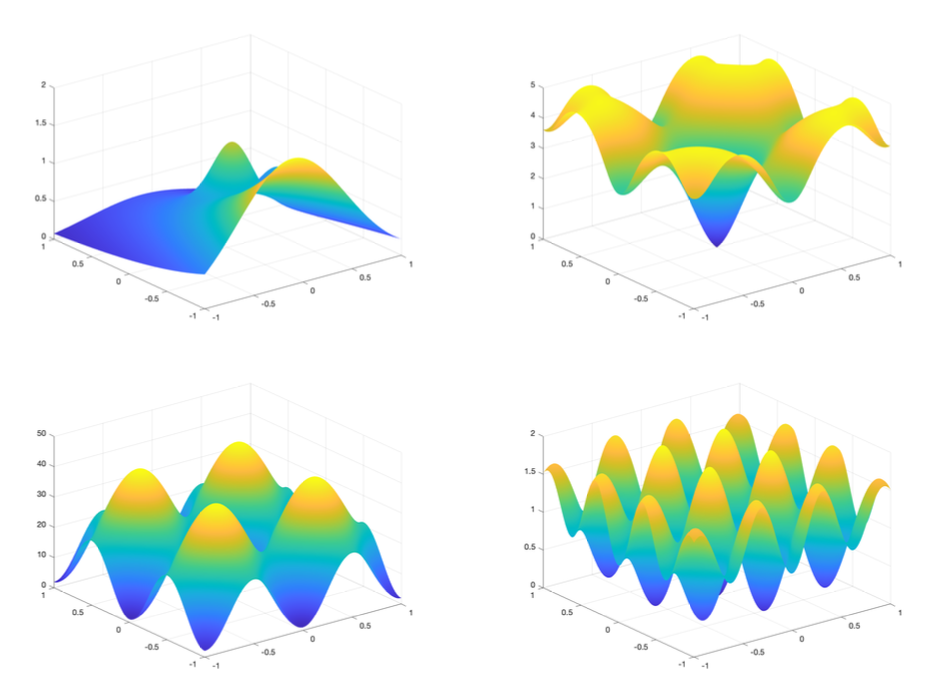} 
     &
     \includegraphics[width=0.47\textwidth]{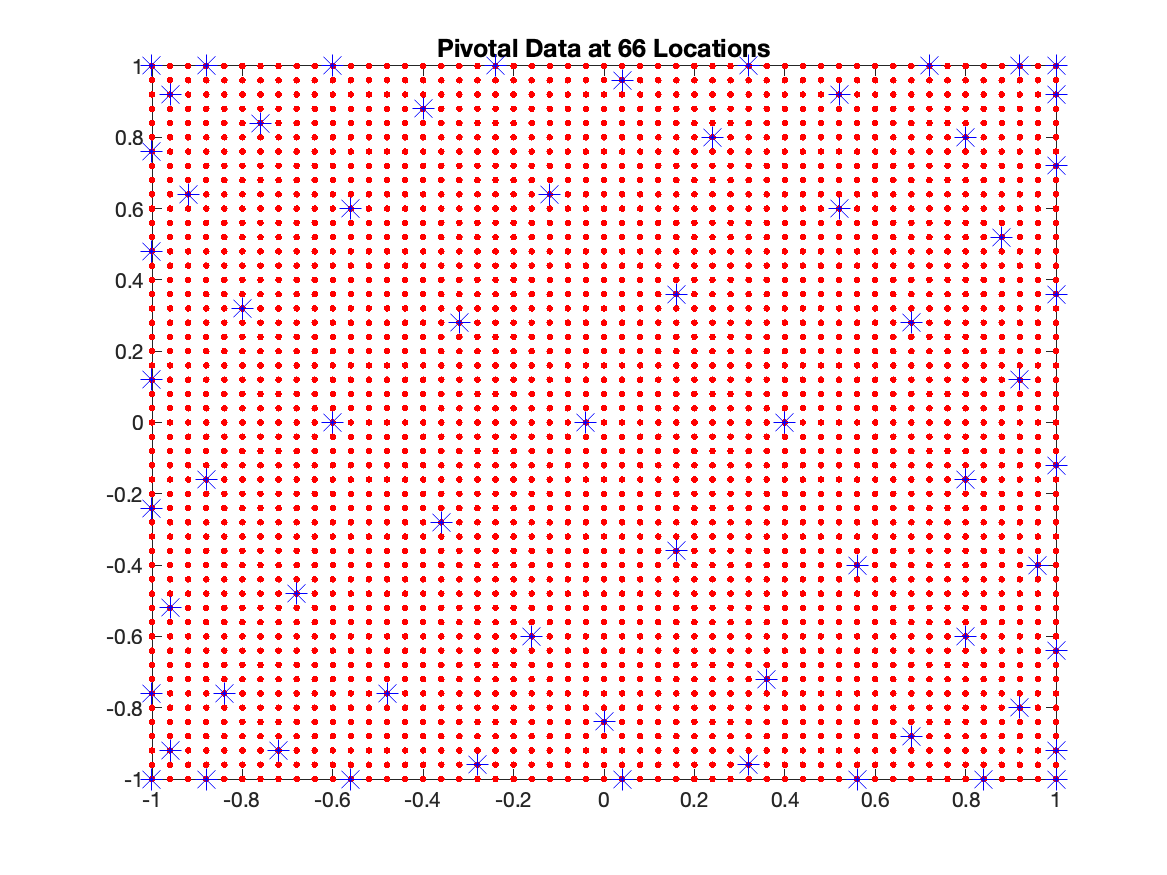} 
\end{tabular}
\vspace{-3mm}
\caption{Left: Plot of some testing functions (top two: Franke and Ackley. bottom two: Rastrigin and wavy). Right: The entire data locations (red) and the pivotal data locations (blue).  \label{pivotalpoints}}
\end{figure}

We uniformly sampled $51^2$ points over $[-1,1]^2$ to solve for an estimated $\hat{\mathbf{c}}$ of $\mathbf{c}$ of the discrete least squares problem (\ref{lsqmin}). We report the relative error between the estimated $\hat{f}$ and the exact function $f$ in Table~\ref{relTable2D}. The relative errors are computed based on $501^2$ points over $[-1,1]^2$.

Alternatively, we can use our matrix cross approximation technique to select some specific rows and therefore solve a smaller discrete least squares problem. It is worthy noting that our data matrix $[\phi_i(\mathbf{x}_j)]_{i,j}$ in this setup is of full column rank, therefore the second term on the right side of (\ref{lsqIneq}) is zero and the maximal volume algorithm is going to select some specific rows together with all the columns. Let us call these selected rows the pivotal rows, or pivotal sample locations, we can solve (\ref{lsqmin}) only based on those $\mathbf{x}_j$'s which belong to the pivotal sample locations to get an estimated $\hat{f}_p$ of $f$. Right side of Figure~\ref{pivotalpoints} shows the pivotal locations of the uniform sampling. 
We also report the relative error between $\hat{f}_{p}$ and the exact testing
function $f$ in Table~\ref{relTable2D}. Again the relative errors are computed  based on  $501^2$ points over $[-1,1]^2$.  

We can see that the two columns of relative errors in Table~\ref{relTable2D} 
are in the same magnitude, which shows that we just need to use the data at pivotal locations in order to have a good approximation instead of using the entire data $51^2$ locations. That is, our computational method for the least squares solution reveals the pivotal locations in $[-1, 1]^2$ for measuring the unknown target
functions so that we can use fewer measurements and quickly recover them.   

\begin{table}[t]
\caption{Relative errors of least squares approximations for functions of two variables by using $51^2$ sample locations and by using pivotal data locations \label{relTable2D}}
\centering
\vspace{3mm}
\begin{tabular} {c|c|c}
\toprule
$f(x,y)$ & $\|f-\hat{f}\|_{rel}$ & $\|f-\hat{f}_p\|_{rel}$  \\ 
\midrule
$e^{x^2+y^2}$ & 1.93E-05 & 4.59E-05 \\
$\sin(x^2+y^2)$ & 2.13E-05 & 5.07E-05 \\
$\cos(x^2+y^2)$ & 1.28E-05 & 2.83E-05 \\
$\ln(1+x^2+y^2)$ & 1.06E-04 & 2.10E-04 \\
$(1+x^4+y^4)/(1+x^2+y^2)$ & 3.40E-04 & 6.57E-04 \\
Franke's function & 5.90E-02 & 8.10E-02 \\
Ackley's function & 2.10E-02 & 4.05E-02 \\
Rastrigin's function & 7.65E-04 & 1.10E-03 \\ 
Wavy function & 7.03E-02 & 9.80E-02 \\
\bottomrule
\end{tabular}
\end{table}


\section{Conclusion and future research}
In this work, we proposed a family of greedy based maximal volume algorithms which improve the classic maximal volume algorithm in terms of the error bound of cross approximation and the computational efficiency. The proposed algorithms are shown to have theoretical guarantees on convergence and are demonstrated to have good performances on the applications such as image compression and least squares solution.

Future research can be conducted along the following directions. Firstly, we would like to weaken the underlying assumption or seek for a better error bound and an efficient algorithm which can achieve the corresponding bound. Secondly, as data matrices in the real-world applications are usually incomplete or not fully observable, it is vital to generalize the algorithms to the setting which can fit in the case where matrices have missing values. Thirdly, if a matrix is sparse or special structured, we want to develop some more pertinent algorithm to have a better approximation result. Finally, 
we plan to generalize the proposed approach to the setting of tensor  approximation.

\section*{Acknowledgements}
The authors are very grateful to the editor and referees for their helpful comments. In particular, the authors would like to thank one of the referees for bringing up the most recent reference \cite{HH23} to their attention.

\section*{Declarations}

\paragraph{Funding} The second author is supported by the Simon Foundation for collaboration grant \#864439. 

\paragraph{Data Availability} The data and source code that support the findings of this study are available upon request.

\paragraph{Conflict of Interest}
The authors have no conflicts of interest to declare.





\end{document}